\documentclass[11pt, english]{article}
\usepackage{amsthm}
\usepackage{amsmath}
\usepackage{amssymb}
\usepackage{graphicx}

\makeatletter

\AtBeginDocument{\DeclareFontEncoding{T2A}{}{}}

\numberwithin{figure}{section}
\theoremstyle{plain}

\theoremstyle{plain}

\@ifundefined{date}{}{\date{}}
\newtheorem{theorem}{Theorem}[section]
\newtheorem{predl}{Proposition}[section]
\newtheorem{opr}{Definition}[section]
\newtheorem{lm}{Lemma}[section]

\numberwithin{equation}{section}

\title{Isomorphism between super Yangian and quantum loop superalgebra. I}
\author{V. Stukopin}

\usepackage{babel}
\providecommand{\propositionname}{Proposition}
\providecommand{\theoremname}{Theorem}

\begin{document}

\maketitle 

\begin{abstract}
Following V. Toledano-Laredo and S. Gautam approach we construct isomorphism between super $\hbar$-Yangian  $Y_{\hbar}(A(m,n))$ of special linear superalgebra and quantum loop superalgebra $U_{\hbar}(LA(m,n))$.\\

\\

\end{abstract}

\markright{Isomorphism between super Yangian and quantum loop superalgebra}

\section{Introduction}

Lie Superalgebras are a working tool for physicists in the study of supersymmetric models of quantum field theory. Classification of simple Lie superalgebras of classical type was obtained by Victor Kac in the late 70-ies of the last century  (see \cite{K}). Lie superalgebras of classical type in this classification are divided into basic and strange Lie superalgebras, and the last unlike basic Lie superalgebras, are not contragredient Lie superalgebras (see \cite{K}, \cite{F-S}). In the middle of the 80s of the last century V. Drinfeld had introduced  quantum groups that are deformations of universal algebras of simple Lie algebras (see \cite{Dr}, \cite{Dr1}, \cite{Dr2}, \cite{Ch-Pr}).  One of the most important examples of such deformations were the Yangians (\cite{Dr}, \cite{Ch-Pr}, \cite{Mol}) associated with rational solutions of the quantum Yang-Baxter equation as well as quantized universal enveloping Kac-Moody algebras and quantum loop algebras. Somewhat later, the Yangians of Lie superalgebras (\cite{St}, \cite{N}) were also studied, which also found application in quantum field theory, in particular, in the quantum theory of superstrings (\cite{D-N-W}, \cite{S-T}).


Despite the marked differences between the Yangians and quantum affine algebras, there are many similarities between them and, first of all, a very similar description of finite-dimensional irreducible representations, which indicates a deep connection between these objects. This relationship was clear already V. Drinfeld, who determined both these classes of Hopf algebras. But explicit results that formulate a connection between these objects have been published relatively recently (see \cite{G-TL}, \cite{G-TL1}, \cite{G-TL2}).

In this paper, we following the approach of S. Gautam and V. Toledano-Laredo (see  \cite{G-TL}), we construct an isomorphism between the Yangian of the special linear Lie superalgebra and the quantum loop superalgebra $U_{\hbar}(LA(m, n))$. It should also mention the works \cite{G-TL1}, \cite{G-TL2}, in which this approach is significantly strengthened and used to investigate the connection between categories of representations of Yangians and quantum loop algebras. Actually, such an investigation is our goal, and this work is the first step in this direction (see also the papers \cite{St11}, \cite{St22}, devoted to the representations of the special linear Lie superalgebra of the Yangian).
To simplify the presentation, we also omitted the proofs of some auxiliary assertions, which we are going to publish in the second part of this paper.

A few words about the organization of this work. In the second paragraph, we recall the definitions of the main actors - the Yangian of a special linear Lie superalgebra and the quantum loop superalgebra, and the result on the classification of finite-dimensional irreducible modules over a Yangian and a quantum affine algebra. In the third section we formulate the main result of the paper. In the last, fourth  paragraph, we give a detailed proof of the theorem on isomorphism.



 We'll use the following standard notations. We denote by $\mathbb{C} $ the complex numbers field, $M_n(K)$ be the ring of $N\times N$ matrices with elements from the ring $K$; we denote by $K[u]$, $K[[u]]$ the ring of polynomials, respectively, formal power series, with coefficients in the ring $K$. We denote the end of the proof by the symbol $\square$.

\vspace{1cm}

\section{Super Yangian $Y_{\hbar}(A(m,n))$  and Quantum Loop Superalgebra $U_{\hbar}(LA(m,n))$}\label{s20}

In this section we recall the definitions of the Yangian of a special linear superalgebra and the quantum loop superalgebra $U_{\hbar}(LA(m,n))$. But first, for the convenience of the reader, we recall the definition of a special linear Lie superalgebra $\mathfrak{sl}(m+1, n+1) = A(m,n)$.

\subsection{Special linear Lie superalgebra}

Lie superalgebra $A(m,n)$ is defined by its Cartan matrix $A=(a_{i,j})_{i,j=1}^{m+n+1}$. Its non zero elements has a following form:
$$a_{i,i}=2, \quad a_{i,i+1}=a_{i+1,i}=-1, \quad i < m+1;$$
$$\quad a_{i-1,i} = a_{i,i-1}=1, \quad a_{i,i}=-2, \quad m+1<i, \quad i \in I=\{1, \ldots , m+n+1\}.$$

This Cartan matrix is symmetrizable and it is often more convenient to use the symmetric Cartan matrix. True, in the corresponding symmetric matrix, the diagonal elements starting from $m + 2$-th are negative and equal to $-2$, and the elements on the row of the upper and lower diagonals, starting with $m + 1$-th, are $1$. The remaining elements are exactly the same as those of the selected Cartan matrix.

Thus, in the relations we use the following symmetrized Cartan matrix of the Lie superalgebra $A(m, n)$:

$$A = (a_{i,j})_{i,j=1}^{m+n+1}= \begin{pmatrix} 2& -1&\ldots &0&0& \ldots& 0&0\\-1& 2& \ldots&0&0& \ldots& 0&0\\ \ldots& \ldots& \ldots&\ldots&\ldots& \ldots& \ldots&\ldots\\
0& 0&\ldots& 0& 1& \ldots &0&0 \\0&0&\ldots& 1 & -2& \ldots &0&0\\ \ldots& \ldots&\ldots& \ldots& \ldots& \ldots & \ldots & \ldots\\
0& 0&\ldots& 0& 0& \ldots& 1& -2  \end{pmatrix}.$$ 

The Lie superalgebra $\mathfrak{g} = A(m,n)$ is generated by the generators: $h_i, x^{\pm}_i$,  $i \in I$.  The  generators $x^{\pm}_{m + 1}$ are odd, while the remaining generators are even, that is, the parity function $p$ takes on the following values: $p(h_i) = 0, i \in I, p(x^{\pm}_j) = 0, j \neq m + 1,  p(x^{\pm}_{m + 1}) = 1$. These generators satisfy the following defining relations:

\begin{eqnarray}
&[h_i, h_j] = 0, \quad \\
&[h_i, x_j^{\pm}] = \pm a_{ij}x_j^{\pm}, \quad\\
&[x_i^+, x_j^-] = \delta_{ij} h_i, \quad\\
&[[x^{\pm}_{m},x^{\pm}_{m+1}], [x^{\pm}_{m+1}, x^{\pm}_{m+2}]] = 0,\\
&ad^{1- \tilde{a}_{ij}}(x^{\pm}_i)x^{\pm}_j = [x_i^{\pm},[x_i^{\pm}, x_j^{\pm}]] = 0. \quad
\end{eqnarray}

We note that the last two relations are called Serre relations.

\vspace{0.5cm}

\subsection{Definition of super Yangian }

Let $I = \{1, 2, \ldots, m+n+1\}$ the set of index numbers of simple roots of Lie superalgebra $A(m,n)$.

\begin{opr} \label{opr2.4.2}
Yangian (more precisely, $\hbar$-Yangian) $Y_{\hbar}(\mathfrak{g})$ of Lie superalgebra $ \mathfrak{g}$ is a Hopf superalgebra over ring $\mathbb{C}[[\hbar]]$ og formal power series, generated as associative superalgebra by generators $h_{i,k}:= h_{\alpha_i, k}, x^{\pm}_{i,k}:=x^{\pm}_{\alpha_i,k}, i \in I, k \in \mathbb{Z}_+$,

which satisfy the following system of defining relations:

\begin{eqnarray}
&[h_{i,k}, h_{j,l}] = 0,\quad \label{Y1}\\
&[h_{i,0}, x^{\pm}_{j,s}] = \pm d_i a_{ij}x^{\pm}_{j,s}, \quad \label{Y2}\\
&\delta_{i,j} h_{i,k+l} = [x_{i,k}^+, x_{j,l}^{-}], \quad \label{Y5} \\
&[h_{i,k+1},x_{j,l}^{\pm}]= [h_{i,k}, x_{j,l+1}^{\pm}] + \frac{d_ia_{ij}\hbar}{2}(h_{i,k}x_{j,l}^{\pm} + x_{j,l}^{\pm} h_{i,k}),\quad  i\ \mbox{or} \ j\neq m+1, \quad \label{Y3} \\
&[h_{m+1,k+1},x_{m+1,l}^{\pm}] = 0, \quad \label{Y2'} \\
&[x_{i,k+1}^{\pm},x_{j,l}^{\pm}] = [x_{i,k}^{\pm}, x_{j, l+1}^{\pm}] + \frac{d_ia_{ij}}{2} \hbar(x_{i,k}^{\pm}x_{j,l}^{\pm} + x_{j,l}^{\pm}x_{i,k}^{\pm}),\quad
i\neq m\ \mbox{or} \ j\neq m+1, \qquad \label{Y4} \\
&[x_{m+1,k+1}^{\pm}, x_{m+1,l}^{\pm}] = 0, \quad \label{Y4'}\\
&\sum_{\pi \in \mathfrak{S}_r} [x^{\pm}_{i, t_{\pi(1)}},[x^{\pm}_{i, t_{\pi(2)}}, \ldots, [x^{\pm}_{i, t_{\pi(r)}}, x^{\pm}_{j, s}]\ldots] = 0, \quad i \neq j, \quad r= 1 - \tilde{a}_{ij}, \quad \label{Y6} \\
&[[x_{m,k}^{\pm}, x_{m+1,0}^{\pm}], [x_{m+1,0}^{\pm}, x_{m+2,t}^{\pm}]] = 0, \quad
\label{Y7}
\end{eqnarray}
for all integer $m$,  $r$, $l$, $t$. Here $\mathfrak{S}_r$ is a permutation group of the  $r$-elements set. The sum in relation (\ref{Y6}) is taken by all permutations  $\sigma$ of set $\{1,\ldots,r\}$.  We note that the relation (\ref{Y6}) takes the following form for $j \in \{i-1, i+1\} $
$$[x_{i,k}^{\pm}, [x_{i,s}^{\pm},x_{j,l}^{\pm}]] + [x_{i,s}^{\pm}, [x_{i,k}^{\pm}, x_{j,l}^{\pm}]] = 0, $$
for arbitrary natural numbers $k$,  $s$, $l$, since in these cases $\tilde{a}_{ij} = -1$. 
In the remaining  cases for $i \neq j$ we have $\tilde{a}_{ij} = 0$ and relation (\ref{Y6}) is transformed to the relation
$$[x_{i,k}^{\pm}, x_{j,l}^{\pm}] = 0, $$
since  for $i \neq j$ we have $\tilde{a}_{ij} = 0$ in the remaining cases.

The parity function takes the following values on the generators:
$p(x_{j,k}^{\pm})=0,$ for $k \in \mathbb{Z}_+, j \in I\setminus \{m+1\} \quad$  $p(h_{i,k}) = 0$, for  $i \in I, k \in \mathbb{Z}_+$,  $p(x_{m+1, k}^{\pm}) = 1, k \in \mathbb{Z}_+.$
\end{opr}

We note that the universal enveloping superalgebra $U(\mathfrak{g})$ is naturally embedded in $Y(\mathfrak{g})$, and the embedding is given by the formulas $h_i\mapsto h_{i,0},  x^{\pm }_i \mapsto x^{\pm}_{i,0}$. We shall identify the universal enveloping superalgebra $U(\mathfrak{g})$ with its image in Yangian.

\vspace{1.5cm}

\subsection{Definition of quantum loop superalgebra}

In this subsection we define quantum loop superalgebra using a definition of a quantized universal enveloping affine superalgebra in terms of the current system of generators and generating relations, the natural analogue of the new system of generators and defining relations by Drinfeld. Actually, the quantum affine superalgebra defined in this way allows us to naturally separate a quantum loop superalgebra from it, simply by putting in the following defining relations the element $C$, which is the deformation of the central element  $c$ equal to 1.

We now explicitly formulate the definition of the quantum loop superalgebra $U_{\hbar}(LA(m,n))$, slightly changing the notation, following G. Lustig. In addition, we assume that $q= e^{\hbar/2}$.  We use also the following standard notation: 
$$[n]_q = \dfrac{q^n - q^{-n}}{q - q^{-1}}, \quad [n]_q! = [n]_q \cdot [n-1]_q\cdot \ldots [1]_q, \quad \left[n \atop k \right]_{q} = \dfrac{[n]_q}{[n-k]_q \cdot [k]_q}.$$

Now, we define quantum universal enveloping superalgebra of Lie superalgebra  $\mathfrak{g} = A(m,n)$.

\begin{opr}
Let $U_{\hbar}(L\mathfrak{g})$ be an associative superalgebra over  $C[[\hbar]]$ topologically generated by generators  $\{E_{i,k}, F_{i,k}, H_{i,k}\}_{i \in I, k \in \mathbb{Z}}$, which satisfy the following system of defining relations:\\
1) For every  $i, j \in I$ and $r,s \in \mathbb{Z}$
\begin{equation}
[H_{i,r}, H_{j,s}] = 0. \label{QL1}
\end{equation}
2) For every $i, j \in I$ and $k \in \mathbb{Z}$
\begin{equation}
[H_{i,0}, E_{j,k}] = a_{i,j}E_{j,k}, \quad [H_{i,0}, F_{j,k}] = -a_{i,j}F_{j,k}. \label{QL2}
\end{equation}
3) For every $i, j \in I$ and $r, k \in \mathbb{Z}\backslash\{0\}$
\begin{equation}
[H_{i,r}, E_{j,k}] = \frac{[ra_{i,j}]_{q_i}}{r}E_{j,r+k}, \quad [H_{i,r}, F_{j,k}] = -\frac{[ra_{i,j}]_{q_i}}{r}F_{j,r+k}. \label{QL3}
\end{equation}
4) For every $i, j \in I$ and $k, l  \in \mathbb{Z}$
\begin{eqnarray}
&E_{i,k+1}E_{j,l} - q_i^{a_{ij}}E_{j,l}E_{i,k+1} = q_i^{a_{ij}}E_{i,k}E_{j,l+1} - E_{j,l+1}E_{i,k},  \nonumber \quad \\
&F_{i,k+1}F_{j,l} - q_i^{-a_{ij}}F_{j,l}F_{i,k+1} = q_i^{-a_{ij}}F_{i,k}F_{j,l+1} - F_{j,l+1}F_{i,k}. \quad \label{QL4}
\end{eqnarray}
5) For every $i, j \in I$ and $k, l  \in \mathbb{Z}$
\begin{equation}
[E_{i,k}, F_{j,l}] = \delta_{i,j} \frac{\psi_{i, k+l} - \varphi_{i, k+l}}{q_i - q_i^{-1}}. \label{QL5}
\end{equation}
6) Let  $i \neq j \in I$ and let  $m = 1 - a_{ij}$. For every $k_1, \ldots, k_m \in \mathbb{Z}$ and $l \in \mathbb{Z}$
\begin{equation}
\sum_{\pi \in \mathfrak{S}_m}\sum_{s=0}^m (-1)^s \left[m \atop s \right]_{q_i} E_{i,k_{\pi(1)}}\cdot \ldots \cdot E_{i,k_{\pi(s)}}\cdot E_{j,l}\cdot E_{i,k_{\pi(s+1)}} \cdot \ldots \cdot E_{i,k_{\pi(m)}} = 0,
\end{equation}

\begin{equation}
\sum_{\pi \in \mathfrak{S}_m}\sum_{s=0}^m (-1)^s \left[m \atop s\right]_{q_i} F_{i,k_{\pi(1)}}\cdot \ldots \cdot F_{i,k_{\pi(s)}}\cdot F_{j,l}\cdot F_{i,k_{\pi(s+1)}} \cdot \ldots \cdot F_{i,k_{\pi(m)}} = 0, \label{QL6}
\end{equation}

7) 
\begin{equation}
[[E_{m,k}, E_{m+1,0}]_q, [E_{m+1,0}, E_{m+2,r}]_q]_q = 0,  \label{QL7}
\end{equation}

\begin{equation}
[[F_{m,k}, F_{m+1,0}]_{q^{-1}}, [F_{m+1,0}, F_{m+2,r}]_{q^{-1}}]_{q^{-1}} = 0,  \label{QL71}
\end{equation}

where elements  $\psi_{i,r}, \varphi_{i,r}$ are defined the following formulas:
$$\psi_i(z) = \sum_{r \geq 0} \psi_{i,r}z^{-r} = \exp(\frac{\hbar d_i}{2}H_{i,0})\exp((q_i - q_i^{-1})\sum_{s \geq 1} H_{i,s} z^{-s}),$$
$$\varphi_i(z) = \sum_{r \geq 0} \varphi_{i,-r}z^{r} = \exp(-\frac{\hbar d_i}{2}H_{i,0})\exp(-(q_i - q_i^{-1})\sum_{s \geq 1} H_{i,-s} z^{s}),$$
and  $\psi_{i,-k} = \varphi_{i,k} = 0$ for $k \geq 1$. In addition, $p(H_{i,r}) = 0$ for $i \in I, r \in \mathbb{Z}_+$, and $p(X^{\pm}_{i,r}) = 0$ for $i \in I\backslash\{m+1\}, r \in \mathbb{Z}$,  $p(x^{\pm}_{m+1, r}) = 0$  for $r \in \mathbb{Z}$. We also use the notation   $[a,b]_q := ab - (-1)^{p(a)p(b)} q ba$ for $q$-(super)commutator.
\end{opr}

We also denote by $U^0\subset U_{\hbar}(L\mathfrak{g})$ a commutative subalgebra generated by generators $\{H_{i,r}\}_{i \in I, r \in \mathbb{Z}}$.

\subsection{Representation theory of super Yangian and quantum loop superalgebra}

For the reader's convenience, we formulate the main result of the paper \cite{St11}, the classification theorem for finite-dimensional irreducible representations of the Yangian $Y(A(m,n))$. (We note that $Y(\mathfrak{sl}(m+1,n+1)) = Y(A(m,n))$ for $m \neq n$. We give its obvious modification for $Y_{\hbar}(A(m, n)) $

\begin{theorem} \label{theorem01}
1) Every irreducible finite-dimensional  $Y_{\hbar}(A(m,n))$-module  $V$ is a module with highest weight $d$ :$V = V(d)$, i. e., 
$$h_i(u)v_0  = \left(1 +  \hbar\sum_{k=0}^{\infty} h_{i,k}\cdot u^{-k-1}\right)v_0 =  \left(1 +  \hbar\sum_{k=0}^{\infty} d_{i,k}\cdot u^{-k-1}\right)v_0 , $$
where $v_0$ is a highest vector, and $i = \{1, 2, \ldots, m+n+1\}$.\\
2) The module $V(d)$ is finite-dimensional if and only if there exist polynomials $P^d_i$,  $i \in \{1,2, \ldots, m, m+2, \ldots m+n+1\} = I \backslash \{m+1\}$, as well as polynomials  $P^d_{m+1}, Q^d_{m+1}$, which satisfy the following conditions:\\
a) all these polynomials with leading coefficients equal to 1 (or monic);\\
b) \begin{eqnarray}
&\dfrac{P^d_{i}(u+ d_ia_{ii}\hbar/2)}{P^d_{i}(u)} = 1 + \hbar\sum_{k=0}^{\infty} d_{i,k}\cdot u^{-k-1}, \quad i \in I \backslash \{m+1\}, \qquad \label{Y3.1} \\
& \dfrac{P^d_{m+1}(u)}{Q^d_{m+1}(u)} = 1 + \hbar\sum_{k=0}^{\infty} d_{m+1,k}\cdot u^{-k-1}. \label{Y3.2} \qquad
\end{eqnarray}
Here $d_i a_{ii}$ is the matrix element of the symmetrized Cartan matrix of the Lie superalgebra $A(m, n)$.
\end{theorem}

We now formulate an analogue of the above result on the classification of finite-dimensional irreducible representations of the Yangian for the quantum affine Lie superalgebra $U_q(A^{(1)}(m,n))$. The proof of this theorem will be published separately.

\begin{theorem} \label{theorem02}
1) Every irreducible finite-dimensional $U_q(A^{(1)}(m,n))$-module $V$ is a module with highest weight  $\delta$ :$V = V(\delta)$, i.e.
$$\psi_i(z)v_0  = \left(\sum_{k=0}^{\infty} \delta^+_{i,k}\cdot z^{-k}\right)v_+, \quad  \varphi_i(z)v_0 = \left(\sum_{k=0}^{\infty} \delta^-_{-i,k}\cdot z^{k}\right)v_+, $$
where  $v_0$ is a highest vector and $i = \{1, 2, \ldots, m+n+1\}$.\\
2) The module $V(\delta)$ is finite-dimensional if and only if there exist polynomials $P^{\delta}_i,$ $i \in \{1,2, \ldots, m, m+2, \ldots m+n+1\} = I \backslash \{m+1\},$ as well as polynomials $P^{\delta}_{m+1}, Q^{\delta}_{m+1}$,   which satisfy the following conditions:\\
a) all these polynomials with leading coefficients equal to 1 and non-zero free terms;\\
b) \begin{eqnarray}
&q^{-d_i a_{ii}/2}\dfrac{P^{\delta}_{i}(q^{d_ia_{ii}}z)}{P^{\delta}_{i}(z)} = \sum_{k=0}^{\infty} \delta^+_{i,k}\cdot z^{-k} = \sum_{k=0}^{\infty} \delta^-_{i,-k}\cdot z^{k}, \quad i \in I \backslash \{m+1\}, \qquad \label{Q3.1}\\
& \dfrac{P^{\delta}_{m+1}(z)}{Q^{\delta}_{m+1}(z)} =  \sum_{k=0}^{\infty} \delta^+_{m+1, k}\cdot z^{-k} = \sum_{k=0}^{\infty} \delta^-_{m+1,-k}\cdot z^{k}. \qquad \label{Q3.2} 
\end{eqnarray}
\end{theorem}

\vspace{1cm}

\section{Main result}\label{s2}

We denote (as above) by $L\mathfrak{g}$ the Lie algebra (Laurent polynomial) loops with values in the simple algebra (basic Lie superalgebra) of $\mathfrak{g}$. In this subsection we construct a homomorphism of the quantum loop (super)algebra $ U_{\hbar}(L\mathfrak{g})$ into the quantum (super)algebra $Y_{\hbar}(\mathfrak{g})$ from which the Yangian $Y(\mathfrak{g})$ is obtained by specialization for $\hbar = 1$. In constructing, we confine ourselves to the special case of a basic Lie superalgebra of type $A(m,n)$, that is, we assume that $\mathfrak{g} = A(m,n)$. To construct this homomorphism, we need descriptions of quantum affine algebras (superalgebras) and Yangians in terms of generating functions of generators. 

Let $I = \{1,2, \ldots, m + 1, \ldots m+n+1\} $. This set will be identified with the set of simple roots $\{\alpha_1, \alpha_2, \ldots, \alpha_{m}, \alpha_{m+1}, \ldots \alpha_{m+n+1}\}$ base we assume that $\alpha_{m+1}$ is an odd root, that is, we deal with a distinguished system of simple roots of the basic Lie superalgebra $A(m,n)$. 

Let $\{E_{i, r}, F_{i, r}, H_{i, r} \}_{i \in I, r \in \mathbb{Z}}$ be the loop generators of the quantum affine algebra $U_{\hbar}(L\mathfrak{g})$, and $\{e_{i, k}, f_{i, k}, h_{i, k}\}_{i \in I, k \in \mathbb{Z}_+} $ are the generators of the Yangian $Y_{\hbar}(\mathfrak{g})$.

Let $\mathfrak{g} = A(m,n)$, $\widehat{Y_{\hbar}(\mathfrak{g})}$ be a completion of Yangian $Y_{\hbar}(\mathfrak{g})$ with respect to $h$-adic topology, defined by natural filtration. I recall, that this filtration defined by $N$-grading, which, is defined as follows:
$$deg(h_{i,k}) = deg(x^{\pm}_{i,k}) =k, deg(\hbar) = 1. $$ 

Define the map

\begin{equation}
\Phi : U_{\hbar}((L\mathfrak{g}) \rightarrow  \widehat{Y_{\hbar}(\mathfrak{g})},\label{7261}
\end{equation}
on generators by the following formulas:

\begin{eqnarray}
&\Phi(H_{i,r}) = \dfrac{\hbar}{q_i - q_i^{-1}}\sum_{k \geq 0} t_{i,k} \frac{r^k}{k!}, \quad \label{72620}\\
&\Phi(E_{i,r}) =  e^{r \sigma_i^+}\sum_{m \geq 0} g_{i,m}^+ e_{i, m}, \quad \label{72621}\\
&\Phi(F_{i,r}) =  e^{r \sigma_i^-}\sum_{m \geq 0} g_{i,m}^- f_{i, m}. \quad \label{72622}
\end{eqnarray}

Here, as above, we use the following notations $q = e^{\hbar/2}$, $q_i = q^{d_i}$, $d_i$ are elements  of the symmetrizing matrix for the Cartan matrix  Lie (super)algebra $\mathfrak{g} = A(m,n)$. We use also the system of logarithmic generators generators $\{t_{i,r} \}_{i \in I, r \in \mathbb{N}}$ of commutative subalgebra $Y_{\hbar}(\mathfrak{h}) \subset Y_{\hbar}(\mathfrak{g})$ generated by generators $\{ h_{i,r} \}_{i \in I, r \in \mathbb{N}}$. These logarithmic generators are defined by the following equality for the generating functions:

\begin{equation}
\hbar \sum_{r \geq 0} t_{i,r} u^{-r-1} = \log\left(1 + \sum_{r \geq 0} h_{i,r} u^{-r-1}\right). \label{7.1.1}
\end{equation}

The elements $\{g^{\pm}_{i,m}\}_{i \in I, m \in \mathbb{N}}$ lie in the completion of the algebra $Y_{\hbar}(\mathfrak{h})$ and are defined as follows. Consider the following formal power series:

$$\gamma_i(v) = \hbar \sum_{r \geq 0} \dfrac{t_{i,r}}{r!}\left(-\frac{d}{dv}\right)^{r+1}G(v).$$

Then 

$$\sum_{m \geq 0} g^{\pm}_{i,m} v^m = \left(\dfrac{\hbar}{q_i - q_i^{-1}}\right)^{1/2} \exp\left(\frac{\gamma_i(v)}{2}\right). $$

Finally, $\sigma_i^{\pm}$ are homomorphisms of sub-superalgebras
$$\sigma_i ^{\pm}: Y_{\hbar}(\mathfrak{b}_{\pm})(\subset Y_{\hbar}(\mathfrak{g})) \rightarrow Y_{\hbar}(\mathfrak{b}_{\pm}),$$
which are defined on the generators $\{h_{i, r}, e_{i,r}, f_{i,r}\}$ as follows. They leave the generators $h_{i,k}$ fixed, and the other generators act as shifts:
$$\sigma^+_i: e_{j,r} \rightarrow e_{j, r + \delta_{ij}}, \quad \sigma^-_i: f_{j,r} \rightarrow f_{j, r + \delta_{ij}}.$$

The elements  $\{g^{\pm}_{i,m} \}_{i \in I, m \in \mathbb{N}}$   belong to the completion of the algebra $Y_{\hbar}(\mathfrak{h})$ and are defined as follows. Consider the following formal power series: 
$$G(v) = \log\left(\dfrac{v}{e^{v/2} - e^{v/2}}\right) \in Q[[v]]$$
and define  $\gamma_i \in \hat{Y^0[v]}$ by formula:

$$\gamma_i(v) = \hbar \sum_{r \geq 0} \dfrac{t_{i,r}}{r!}\left(-\frac{d}{dv}\right)^{r+1}G(v).$$

Then 

$$\sum_{m \geq 0} g^{\pm}_{i,m} v^m = \left(\dfrac{\hbar}{q_i - q_i^{-1}}\right)^{1/2} \exp\left(\frac{\gamma_i(v)}{2}\right). $$

\vspace{1cm}

Let $\sigma_i^{\pm}$ subsuperalgebra homomorphisms
$$\sigma_i^{\pm}: Y_{\hbar}(\mathfrak{b}_{\pm}) (\subset Y_{\hbar}(\mathfrak{g})) \rightarrow Y_{\hbar}(\mathfrak{b}_{\pm}),$$
which are given on the generators $\{h_{i,r}, e_{i,r}, f_{i,r} \}$ as follows.  They leave the generators $h_{i,k}$ fixed, and to the other generators act as shifts: $\sigma^+_i: e_{j,r} \rightarrow e_{j, r + \delta_{ij}}$, $\sigma^-_i: f_{j,r} \rightarrow f_{j, r + \delta_{ij}}.$

Let $\mathfrak{g} = A(m,n)$. Let also  $\widehat{Y_{\hbar}(\mathfrak{g})}$ be the completion of the Yangian with respect to its $N$-grading. Consider the following mappings 
\begin{equation}
c: U_{\hbar}(L\mathfrak{g}) \rightarrow U(L\mathfrak{g}), 
\end{equation}
which given by $\hbar \rightarrow 0$, and also map  
\begin{equation}
d: U(L\mathfrak{g}) = U(\mathfrak{g}[z, z^{-1}]) \rightarrow  U(\mathfrak{g}),
\end{equation}  
which given by formula  $d(f(z)) = f(1)$. Then the kernel of the composition of these mappings
$$e = d \circ c : U_{\hbar}(L\mathfrak{g}) \rightarrow U(\mathfrak{g}) $$
we denote by $I:= Ker (e)$. We define the completion of a quantum loop algebra with respect to the $I$-adic topology given by the powers of the ideal $I$: 

\begin{equation}
\widehat{U_{\hbar}(L\mathfrak{g})} : = lim_{n \rightarrow \infty} U_{\hbar}(L\mathfrak{g})/I^n.
\end{equation}

We state the main result of this paper.

\begin{theorem} \label{th:main}
The map  $\Phi$, given by the formulas (\ref{72620}) -- (\ref{72622}), defines a monomorphism of associative superalgebras over  $\mathbb{Q}[[\hbar]]$:
\begin{equation}\label{eq:mr}
\Phi :U_{\hbar}(L\mathfrak{g}) \rightarrow \widehat{Y_{\hbar}(\mathfrak{g})},
\end{equation}
which extends to an isomorphism of completions
\begin{equation}\label{eq:mr1}
\widehat{\Phi} : \widehat{U_{\hbar}(L\mathfrak{g})} \rightarrow \widehat{Y_{\hbar}(\mathfrak{g})}.
\end{equation}
\end{theorem}

\vspace{1cm}

\section{Proof of the main result}\label{s3}

The proof decomposes into several auxiliary statements.

\subsection{The Lie superalgebra $\mathfrak{sl}(1,1)$ case}

We first consider an important special case of the Lie superalgebra $\mathfrak{sl}(1,1)$. 

It is easy to see that the Yangian $Y(\mathfrak{sl}(1,1))$ is generated by the generators $h_n, e_n, f_n, n \in \mathbf{Z}_0 $, which satisfy the following defining relations: 
\begin{eqnarray}
&[h_k,h_l] = 0, \quad \\
&[h_k, e_l] = [h_k, f_l] = 0, \quad \\
&[e_k, e_l] = [f_k, f_l] = 0, \quad \\
&[e_k, f_l] = h_{k+l}.  \quad \label{Yp5}
\end{eqnarray}

Similarly, the quantum loop algebra $U_{\hbar}(L\mathfrak{sl}(1,1))$ is generated by the generators $\{E_{n}, F_{n}, H_{n}\}_{n \in \mathbb{Z}}$, which satisfy the following system of defining relations:

\begin{eqnarray}
&[H_r, H_s] = 0, \quad \\
&[H_r, E_s] = [H_r, F_s] = 0, \quad \\
&[E_r, E_s] = [F_r, F_s] = 0, \quad \\
&[E_r, F_s] =   \dfrac{\psi_{r+s} - \varphi_{r+s}}{e^{\hbar/2} - e^{-\hbar/2}}, \quad \label{QLp5}
\end{eqnarray}

for all  $r,s \in \mathbb{Z}$. Here as above elements $\psi_r, \varphi_r$ are defined by the following formulas:
$$\psi_i(z) = \sum_{r \geq 0} \psi_rz^{-r} = \exp(\frac{\hbar}{2}H_0)\exp((e^{\hbar/2} - e^{-\hbar/2})\sum_{s \geq 1} H_s z^{-s}),$$
$$\varphi_i(z) = \sum_{r \geq 0} \varphi_{r}z^{r} = \exp(-\frac{\hbar}{2}H_{0})\exp(-(e^{\hbar/2} - e^{-\hbar/2})\sum_{s \geq 1} H_{-s} z^{s}).$$

First, we describe homomorphism  $\Phi: U_{\hbar}(\mathfrak{sl}(1,1)) \rightarrow Y_{\hbar}(\mathfrak{sl}(1,1))$.

We first recall the definition of the Borel transformation (more precisely, the inverse of the Borel transform), which is a discrete analog of the Laplace transform. We denote by $B$ the transformation associating the functions $f(u) \in A[[u]]$ with values in some associative algebra $ A $, the function $B(f)(v) \in u^{-1}A[[u^{-1}]]$, defined as
\begin{equation}\label{def:Bor}
f(u)= \sum_{k=0}^{\infty}f_ku^{-k-1} \rightarrow B(f)(v) =  \sum_{r=0}^{\infty} \dfrac{f_r v^r}{r!}.
\end{equation}
We note the following properties of the Borel transform of the generating function $t(u)$ of the logarithmic generators $\{t_k\}_{k=0}^{\infty}$.

Because the,
$$[h(u), e_k] = [h(u), f_k] = 0, $$
then it follows that u
$$[t(u), e_k] = [t(u), f_k] = 0, $$
which means that 

\begin{equation}
[B(t(u))(v), e_k] = [B(t(u))(v), f_k] =0.
\end{equation}

Easy to check that  
\begin{equation} \label{eq:Bor_cond}
B(\log(1 - pu^{-1})) = \dfrac{1 - e^{pv}}{v}.
\end{equation}

Indeed, using the Borel transformation properties, which coincide with the properties of the Laplace transform, we obtain the following equalities:
$$B(\log(1 - pu^{-1})) = \frac{1}{v} B(\frac{d}{du}(\log(1 - p u^{-1}))) = \frac{1}{v} B\left(\dfrac{-pu^{-2}}{1 - pu^{-1}}\right)  $$
$$ = \frac{1}{v} B\left(\dfrac{-p}{u(u - p)}\right) = \frac{1}{v} B\left(\dfrac{1}{u} - \dfrac{1}{u - p} \right) = \frac{1}{v}(1 - e^{pv}).$$

To prove the theorem in this particular case, it suffices to verify the equivalence of the relations \ref{Y5}) and (\ref{QL5}).
We first note that
 
\begin{eqnarray}
&\Phi(E_r) = e^{r\sigma_+} \sum_{m \geq 0} g_m e_m = \sum_{m \geq 0} g^{(+,k)}_m e_m = e^{r\sigma_+}g(\sigma_+)e_0, \quad \label{E1} \\
& \Phi(F_r) = e^{r\sigma_-} \sum_{m \geq 0} g_m f_m = \sum_{m \geq 0} g^{(-,k)}_m f_m = e^{r\sigma_-}g(\sigma_-)f_0. \quad \label{E1}
\end{eqnarray}

Let us prove an explicit calculation of equality

\begin{equation} \label{eq:5}
[\Phi(E_r), \Phi(F_l)] = \dfrac{\Phi(\psi_{r+s}) - \Phi(\varphi_{r+s})}{e^{\hbar/2} - e^{\hbar/2}}. 
\end{equation}

Explicitly calculate the right and left sides of the equality to be proved (\ref{eq:5}). It is easy to see that

$$\Phi(E_r)\Phi(F_l) =  e^{r\sigma_+}g(\sigma_+)e_0 e^{r\sigma_-}g(\sigma_-)f_0.$$

As 
$$g(v) = \sum_{m \geq 0} g_m v^m = \sum_{m \geq 0} g_m v^m = \left(\dfrac{\hbar}{e^{\hbar} - e^{-\hbar/2}} \right)^{1/2}\exp(\gamma(v)/2)   $$
$$= \left(\dfrac{\hbar}{e^{\hbar/2} - e^{-\hbar/2}} \right)^{1/2}\exp\left(\frac{1}{2}B(t(u))\left(-\frac{d}{dv}\right) \left(\frac{d}{dv}\log\left(\dfrac{e^{v/2} - e^{-v/2}}{v} \right) \right)\right).  $$

For brevity we use the notation $\partial_v := \dfrac{d}{dv}$. Then 

$$g(v) =  \sum_{m \geq 0} g_m v^m = \left(\dfrac{\hbar}{e^{\hbar/2} - e^{-\hbar/2}} \right)^{1/2}\exp\left(\frac{1}{2}B(t(u))\left(-\partial_v\right) \left(\log\left(\dfrac{e^{v/2} - e^{-v/2}}{v} \right) \right)'\right).  $$

It is easy to check that

$$g(\sigma_+) e_0 = \sum_{m \geq 0} g_m \sigma_+^m e_0 = \sum_{m \geq 0} g_m  e_m.$$ 

Similarly, 

$$g(\sigma_-) f_0 = \sum_{m \geq 0} g_m \sigma_-^m e_0 = \sum_{m \geq 0} g_m  f_m.    $$

Note also that if the variables $u$ and $v$ commute, then 
$$ g(u) g(v) =  \left(\dfrac{\hbar}{e^{\hbar} - e^{-\hbar/2}} \right)\exp(\gamma(u)/2 + \gamma(v)/2) =$$  
$$ \left(\dfrac{\hbar}{e^{\hbar} - e^{-\hbar/2}} \right)\exp\left(\frac{1}{2}B(t(u_1))\left(-\frac{d}{du}\right) (\log(f_0(u)))' + \frac{1}{2}B(t(u_1))\left(-\frac{d}{dv}\right) \frac{d}{dv}(\log(f_0))'\right),$$

where $f_0(u) = \dfrac{e^{u/2} - e^{-u/2}}{u} .$

Now we can calculate $\Phi(E_r)\Phi(E_l)$. Actually, 

$$ \Phi(E_r)\Phi(E_l) =  e^{r\sigma_+}g(\sigma_+)e_0 e^{r\sigma_-}g(\sigma_-)e_0f_0 = g(\sigma_+) g(\sigma_-)e_r f_l $$
Similarly, 

$$ \Phi(F_l)\Phi(E_r) =  g(\sigma_+) g(\sigma_-)f_l e_r. $$

From this we immediately obtain that
$$[\Phi(E_r)\Phi(E_l), \Phi(E_l)\Phi(E_r)] =  [g(\sigma_+) g(\sigma_-)e_r f_l, g(\sigma_+) g(\sigma_-) f_l e_r] $$
$$ =   \sum_{m\geq 0}\sum_{n\geq 0} g_m g_n (e_{m+r}f_{n+l} - f_{n+l}e_{m+r}) = \sum_{m\geq 0}\sum_{n\geq 0} g_m g_n h_{m+n+r+l}$$

On the other hand,

$$\Phi(H_r) = \dfrac{B(t(u))(r)}{e^{\hbar/2} - e^{-\hbar/2}}. $$
Then  

$$\Phi(\psi(z)) = \sum_{r \geq 0} \Phi(\psi_r) z^{-r}= \exp\left(\dfrac{\hbar \Phi(H_0)}{2}\right) \exp\left(\left(e^{\hbar/2} - e^{\hbar/2}\right)\sum_{s \geq 1} \Phi(H_s) z^{-s}\right) $$
$$= \exp\left(\dfrac{\hbar \Phi(H_0)}{2}\right)\exp\left(\sum_{s \geq 1} B(t(u))(s) z^{-s}\right) $$ 
$$=  \exp\left(\frac{\hbar}{2(e^{\hbar/2} - e^{-\hbar/2})}t_0\right)\exp\left(\sum_{s \geq 1} B(t(u))(s) z^{-s}\right). $$

Similarly, 

$$\Phi(\varphi(z)) = \sum_{r \geq 0} \Phi(\varphi_r) z^{r}=  \exp\left(\dfrac{-\hbar \Phi(H_0)}{2}\right)\exp\left(\sum_{s \geq 1} B(t(u))(-s) z^s\right). $$

We describe the action of generators, respectively, of a quantum loop superalgebra and a Yangian, in finite-dimensional irreducible modules. We need the results of the papers \cite{St11} and \cite{St22}, in which the classification of finite-dimensional irreducible Yangian modules of a special linear Lie superalgebra is given. We need an explicit definition of the action of the Cartan generators on the highest vectors of finite-dimensional irreducible modules over, respectively, the Yangian and the quantum loop algebra.

We denote by $D^Y$ the Yangian morphism into a finite-dimensional irreducible module given by a family of Drinfeld polynomials. We now consider the case of Yangian $Y(\mathfrak{sl}(1,1))$. In this case the finite-dimensional irreducible Yangian module is given by a pair of Drinfeld polynomials $P^d(u)$ and $Q^d(u)$ with leading coefficient equal to 1. In this case the polynomials are uniquely determined by their complex roots: $a_1, \ldots, a_n$ and $b_1, \ldots , b_n$. In this case, the action of the generating function of the Yangian on the highest vector is given by the following formula (see \cite{St11}):

\begin{equation} \label{eq:h1}
h(u)v_0 = \left(1 + \sum_{k \geq 0} d_k u^{-k-1}\right) v_0, \quad \dfrac{P^d(u)}{Q^d(u)} = 1 + \sum_{k \geq 0} d_k u^{-k-1}, \quad
\end{equation}   

where as above  $h(u) = 1 + \sum_{k \geq 0} h_k u^{-k-1}$ be a generating function of Cartan generators. In other words

\begin{equation} \label{eq:h2}
h(u)v_0 =  \dfrac{P^d(u)}{Q^d(u)} v_0 = R(m) = \dfrac{(u - a_1)\cdot \ldots (u - a_n)}{(u - b_1)\cdot \ldots (u - b_n))} v_0. 
\end{equation}   

Let's define then  

\begin{equation}\label{eq:DY}
D^Y(h(u) = \dfrac{P^d(u)}{Q^d(u)} = \dfrac{(u - a_1)\cdot \ldots (u - a_n)}{(u - b_1)\cdot \ldots (u - b_n))}. 
\end{equation}


Assuming the quantities  $a_1, \ldots, a_n, b_1, \ldots, b_m$ as parameters, we obtain the mapping

$$D^Y: Y^0 \rightarrow \mathbf{C}[\hbar, a_1, \ldots, a_n, b_1, \ldots, b_n]. $$

Cartan subalgebra of the Yangian into a commutative polynomial ring defined by the formula (\ref{eq:DY}).

Similarly we define the map

$$D^U: U^0 \rightarrow S(m) = \mathbf{C}[q, A_1, \ldots, A_n, B_1, \ldots, B_n], $$

\begin{equation}\label{eq:DU}
D^U(\psi(z)) = \dfrac{P^{\delta}(z)}{Q^{\delta}(z)} = \dfrac{(z - A_1)\cdot \ldots (z - A_n)}{(z - B_1)\cdot \ldots (z - B_n))}. 
\end{equation}

We now compute the action of the homomorphism $D^Y: Y(\mathfrak{sl}(1,1)) \rightarrow End(V_{P,Q})$ onto $t(u)$, as well as the images of the generators $h_k$ and $t_k$ under the action of this mapping. The following proposition holds.

\begin{predl} \label{pr:1}
The following formulas are valid:
\begin{eqnarray}
&D^U(\psi_r) = \sum_{p=1}^n B_p^r (B_p - A_p) \left(\prod_{p' \neq p} \dfrac{B_p - A_{p'}}{B_p - B_{p'}}\right), \quad \label{eq:DUpsi1} \\    
&D^U(\varphi_r) = \sum_{p=1}^n B_p^{-r} (A_p - B_p) \left(\prod_{p' \neq p} \dfrac{B_p - A_{p'}}{B_p - B_{p'}}\right), \quad \label{eq:DUvarphi1} \\   
&D^U((q - q^{-1})H_k) = \frac{1}{p}\sum_{p=1}^n (B^p - A^p), \quad \label{eq:H} \\ 
&D^Y(h_r) = \sum_{p=1}^n b_p^r (b_p - a_p) \left(\prod_{p' \neq p} \dfrac{b_p - a_{p'}}{b_p - b_{p'}}\right), \quad \label{eq:Dh1} \\ 
&D^Y(t_r) = \dfrac{1}{r+1}\sum_{p=1}^n \dfrac{b^{r+1}_p - a^{r+1}_{p}}{\hbar}, \quad \label{eq:Dt1} \\
&D^Y(B(t(u))(v) = \sum_{p=1}^n \dfrac{\exp(b_p v) - \exp(a_p v)}{v}. \quad \label{eq:DB1}
\end{eqnarray}
\end{predl}

\begin{proof}
The proof is carried out by direct calculation, by expansion in powers of the variable $u^{-1}$ of the left and right sides of the equalities to be proved.

\end{proof}

From the sentence \ref{pr:1} it follows easily that

\begin{predl} \label{pr:2}
1) Homomorphisms 

\begin{equation}
D^Y: Y^0 \rightarrow \bigoplus_{n \geq 1} R(n),  \quad D^U: U^0 \rightarrow \bigoplus_{n \geq 1} S(m) 
\end{equation}
are injective.

\end{predl}

Now we can formulate and prove the most important auxiliary assertion

\begin{lm} \label{lm:1}
The following equality holds

\begin{equation}\label{eq:Imp}
\Phi \left(\dfrac{\psi_{k} - \varphi_k}{e^{\hbar/2} - e^{-\hbar/2}}\right)  =   \dfrac{\hbar}{e^{\hbar/2} - e^{-\hbar/2}} e^{kv} \exp(\gamma(v))|_{v^n = h_n}. 
\end{equation}
\end{lm}

\begin{proof}
We first calculate the left-hand side of the  equality.

$$\hbar e^{kv}  \exp(B(-\partial)G'(v)|_{v^n = h_n} = \hbar e^{kv} \exp \left( \sum_{p=1}^n \dfrac{e^{b_p (\partial)} - e^{a_p (\partial)}}{\partial}\right) G(v)|_{v^n = h_n} $$
$$= \hbar e^{kv} \sum_{p=1}^n \exp \left( G(v + b_p) - G(v+a_p) \right)|_{v^n=h_n}$$ 
$$=  \hbar e^{kv} \prod_{p=1}^n \dfrac{v - b_p}{v - a_p} \cdot \dfrac{e^{(v-a_p)/2} - e^{-(v-a_p)/2}}{e^{(v-b_p)/2} - e^{-(v-b_p)/2}}|_{v^n=h_n} .$$

So, we obtain  

$$D^Y\left(\hbar e^{kv} \prod_{p = 1}^n \dfrac{v - b_p}{v - a_p} \cdot \dfrac{e^{(v - a_p)/2} - e^{-(v - a_p)/2}}{e^{(v - b_p)/2} - e^{-(v - b_p)/2}}|_{v^n = h_n} \right). $$

A direct calculation proves the following assertion. 

\begin{predl} \label{fs1}
Let  $F(v) = \sum_{k=0}^{\infty} f_k v^k$. Then  
$$F(v)|_{v^n = h_n} = \sum_{k=0}^{\infty} f_k h_k.$$
\end{predl}

Taking into account the equality (\ref{eq:Dh1}) and the sentence \ref{fs1}, we get that the last equality is

$$D^Y\left(\hbar e^{kv}\exp(\gamma v)|_{v^n=h_n}\right) = D^Y\left(\hbar e^{kv} \prod_{p=1}^n \dfrac{v - b_p}{v - a_p} \cdot \exp(\frac{b_p - a_p}{2}) \cdot \dfrac{e^v - e^{a_p}}{e^v - e^{b_p}}|_{v^n = h_n}\right).$$
Taking into account, that  $\lim_{v\rightarrow b_p}\dfrac{v-b_p}{e^v - e^{b_p}} = e^{-b_p}$ we obtain, that last eqiality is equal 

$$\hbar \sum_{p=1}^n e^{k b_p} \hbar \prod_{p'\neq p} \dfrac{b_p - b_{p'}}{b_p - a_{p'}} \cdot \dfrac{1}{b_p - a_p} \cdot \dfrac{e^{b_p} - e^{a_{p'}}}{e^{b_p} - e^{b_{p'}}} \cdot (e^{b_p} - e^{a_p})  $$
$$ =\hbar \sum_{p=1}^n e^{k b_p}  (e^{b_p} - e^{a_p}) \prod_{p'\neq p} \dfrac{e^{b_p} - e^{a_{p'}}}{e^{b_p} - e^{b_{p'}}}. $$

By the formula (\ref{eq:DUpsi1}), we see that

$$(\dfrac{\hbar}{e^{\hbar/2} - e^{-\hbar/2}})D^U(\psi_r) = \hbar\sum_{p=1}^n B_p^r (B_p - A_p) \left(\prod_{p' \neq p} \dfrac{B_p - A_{p'}}{B_p - B_{p'}}\right), $$
that coincides after replacement  $A_p = e^{a_p}$, $B_p = e^{b_b}$ with $D^Y(\hbar e^{rv}\exp(\gamma v)|_{v^n=h_n})$. Analogously, considering the expansion in a neighborhood of an infinitely, we obtain that $D^Y(\hbar e^{rv}\exp(\gamma v)|_{v^{-n-1}=h_n})$ coincides with $(\dfrac{\hbar}{e^{\hbar/2} - e^{-\hbar/2}})D^U(\varphi_r)$.

Lemma is proved.
\end{proof} 

Easy to see that from lemma  \ref{lm:1} and proposition  \ref{pr:2} follows equality (\ref{eq:5}):
$$[\Phi(E_r), \Phi(F_l)] = \dfrac{\Phi(\psi_{r+s}) - \Phi(\varphi_{r+s})}{e^{\hbar/2} - e^{-\hbar/2}},$$
since, as was proved earlier, by the injectivity of the mapping $D^U$, which follows from the proposition \ref{pr:2},  left-hand side of equality is equal to $e^{kv}\exp(\gamma v)|_{v^n=h_n}$. From the lemma \ref{lm:1} and the injectivity of the map $D^Y$, which also follows from the proposition \ref{pr:2}, it follows that the right-hand side of the equality (\ref{eq:5}) also is equal to $e^{kv}\exp(\gamma v)|_{v^n=h_n}$.

 Thus, the equality (\ref{eq:5}) is proved.

Equality (\ref{eq:5}) easily implies the equivalence of (2) and (3).

\vspace{0.5cm}

\subsection{General case}

We now proceed to the proof of the theorem. We note that the equivalence of the relations (\ref{Y5}) and (\ref{QL5}) in the case when $i = m+1$ is proved above. The case when $i\neq m+1$ is essentially proved in the paper \cite{G-TL}, and we omit its proof.

Namely, let us first show that the relations (\ref{Y2}), (\ref{Y3}) are equivalent to the following relation, written in terms of generating functions of generators and shift operators  $\sigma^{\pm}_j$:
\begin{equation}
[h_i(u), x^{\pm}_{j, s}] = \dfrac{\pm \hbar d_i a_{i,j}}{u - \sigma^{\pm}_j \pm \hbar d_i a_{i,j}}h_i(u)x^{\pm}_{j, s}. \label{7263}
\end{equation}
Here 
$$h_i(u) = 1 + \hbar \sum_{r \geq 0} h_{i, r}\cdot u^{-r-1} \in Y_{\hbar}(\mathfrak{g})[[u^{-1}]]. $$

The right-hand side in (\ref{7263}) is understood as an expansion in powers of $u^{-1}$, and the brackets are understood as a supercommutator. The assertion is proved by direct verification. Let's do this check. For simplicity we write $a = \pm \hbar d_i a_{i,j}/2$. 
We multiply the right-hand side of the relation (\ref{Y3}) by $\hbar u^{-r-1}$ and sum over all $r \geq 0$. We obtain

\begin{equation}
u[h_i(u) - 1 - \hbar u^{-1} h_{i,0}, x^{\pm}_{j, s+1}] - [h_i(u) - 1, x^{\pm}_{j, s+1}] = a\{x^{\pm}_{j,s}, h_i(u) - 1\}, \label{7264}
\end{equation}
where $\{x, h\} = xh + (-1)^{p(h)p(x)}hx$ denotes an anti-supercommutator. Using relations (\ref{Y2}) and $[x, h] = [x,h] + 2(-1)^{p(h)p(x)}hx$ we obtain, that \\
\begin{equation}
(u - \sigma^{\pm}_j + a)[h_i(u), x^{\pm}_{j, s}] = 2a h_i(u) x^{\pm}_{j, s},
\end{equation}
Q.E.D. The converse is also true. It is sufficient to take the coefficients of $u^{-k}$, $k \geq 0$ on the right-hand side of the formula (\ref{7264}) in order to obtain formulas (\ref{Y2}), (\ref{Y3}).

Now, in the same way, we rewrite the relations (\ref{Y4}), (\ref{Y6}) in terms of generating functions of generators. We will use the above notation. For the operator $B \in End(V)$ we denote by $B_{(i)} \in End(V^{\otimes n})$ the operator defined by the formula:
$$B_{(i)} = I^{\otimes i-1} \otimes B \otimes I^{\otimes n-i}.$$  Also, for (super) algebras $A$, we define the map $ad^{(m)}: A^{\otimes m} \rightarrow End(A)$ by formula:
$$ad^{(m)}(a_1 \otimes \ldots \otimes a_m) = ad(a_1)_{(1)}\circ \ldots \circ ad(a_m)_{(m)}.  $$

\begin{predl}
1) The relation (\ref{Y4}) for  $i \neq j$ is equivalent is  to the following relation 
\begin{equation}
A(\sigma^{\pm}_i, \sigma^{\pm}_j)(\sigma^{\pm}_i - \sigma^{\pm}_j \mp a \hbar) x^{\pm}_{i,0}x^{\pm}_{j,0} = A(\sigma^{\pm}_i, \sigma^{\pm}_j)(\sigma^{\pm}_i - \sigma^{\pm}_j \pm a \hbar) x^{\pm}_{i,0}x^{\pm}_{j,0}, \label{7265}
\end{equation}
for an arbitrary  $A(v_1, v_2) \in \mathbb{C}[[v_1, v_2]]$, where $a =  \pm \hbar d_i a_{i,j}/2.$\\
2) The relation (\ref{Y4}) for $i = j$ is equivalent to the fact that for an arbitrary $B(v_1, v_2) \in \mathbb{C}[[v_1, v_2]]$, such that  $B(v_1, v_2) = B(v_2, v_1)$ and we have
\begin{equation}
\mu(B(\sigma^{\pm}_{i,(1)}, \sigma^{\pm}_{j, (2)})(\sigma^{\pm}_{i, (1)} - \sigma^{\pm}_{j, (2)} \mp d_i \hbar) x^{\pm}_{i,0} \otimes x^{\pm}_{j,0} = 0, \label{7266}
\end{equation}
where $\mu: Y_{\hbar}(\mathfrak{g})^{\otimes 2} \rightarrow Y_{\hbar}(\mathfrak{g})$ is multiplication in $Y_{\hbar}(\mathfrak{g})$, and  $\sigma^{\pm}_{i,(j)} := (\sigma^{\pm}_{i})_{(j)}$.\\
3) The relation (\ref{Y6}) for $i \neq j$ is equivalent to the fact that for an arbitrary $A(v_1, v_2, \ldots, v_n) \in \mathbb{C}[v_1, v_2, \ldots, v_n]^{\mathfrak{S}_n}$ the following equality holds
\begin{equation}
ad^{(n)}(A(\sigma^{\pm}_{i,(1)}, \ldots , \sigma^{\pm}_{i,(n)})(x^{\pm}_{i,0})^{\otimes n})x^{\pm}_{j,l} = 0, \label{7267}
\end{equation}
where $n = 1 - d_i a_{i,j}$ for $i \leq m+1 $ and $n =  1 + d_i a_{i,j}$ for $i \geq m +1 $.
\end{predl}

\begin{proof}
1) We note that relation (\ref{Y4})
$$ [x^{\pm}_{i,{r+1}}, x^{\pm}_{j,s}] - [x^{\pm}_{i,r}, x^{\pm}_{j,{s+1}}]  = \pm \frac{d_ia_{ij}\hbar}{2}(x^{\pm}_{i,r}x^{\pm}_{j,s} + x^{\pm}_{j,s}x^{\pm}_{i,r}) $$
can be rewritten in the following form:
\begin{equation}
\sigma^{\pm r}_i \sigma^{\pm s}_j (\sigma^{\pm}_i - \sigma^{\pm}_j \mp a \hbar) x^{\pm}_{i,0}x^{\pm}_{j,0} = \sigma^{\pm r}_i \sigma^{\pm s}_j(\sigma^{\pm}_i - \sigma^{\pm}_j \pm a \hbar) x^{\pm}_{i,0}x^{\pm}_{j,0}, \label{7268}
\end{equation}
which proves item 1. \\

2) If  $i = j$, $a = d_ia_{ii}$, then we note that it follows from formula (\ref{7268}) that
$$\mu(\sigma^{\pm r}_{i,(1)} \sigma^{\pm s}_{j, (2)})(\sigma^{\pm}_{i, (1)} - \sigma^{\pm}_{j, (2)} \mp d_i \hbar) x^{\pm}_{i,0} \otimes x^{\pm}_{j,0} = \mu(\sigma^{\pm s}_{i,(1)} \sigma^{\pm r}_{j, (2)})(\sigma^{\pm}_{i, (2)} - \sigma^{\pm}_{j, (1)} \mp d_i \hbar) x^{\pm}_{i,0} \otimes x^{\pm}_{j,0}.$$
We note that this relation is equivalent to the relation
$$\mu(\sigma^{\pm r}_{i,(1)} \sigma^{\pm s}_{j, (2)} + \sigma^{\pm s}_{i, (1)} \sigma^{\pm s}_{j, (2)})(\sigma^{\pm}_{i,(1)} - \sigma^{\pm}_{j, (2)} \mp d_i \hbar) x^{\pm}_{i,0} \otimes x^{\pm}_{j,0} =  0,$$
which proves item 2.\\

3) Note that (\ref{7267}) is just a rewritten identity (\ref{Y4}), and this fact proves  3). We note that the proof of this relation in our case coincides with the proof for the Yangians of simple Lie algebras and quantum loop algebras given in detail in the paper \cite{G-TL}.

Proposition is proved.
\end{proof}

Consider in $Y^0$, the system of generators $\{t_{i,r}\}_{i \in I, r \in \mathbb{Z}_+}$ introduced earlier. This system of generators is, as is easy to see, homogeneous and  $deg(t_{i,r}) = r$. Moreover, $t_{i,0} = h_{i,0}$ and $t_{i,r} \equiv h_{i, r} mod  \hbar$ for an arbitrary  $r \geq 1$ as 
$$t_i(u) \equiv  (\hbar \sum_{r \geq 0} h_{i, r} u^{-r-1}) mod \hbar^2.$$
We consider, as above, $B(t_i(u))(v)$  the inverse Borel transformation  $B$ (see (\ref{def:Bor})) of formal power series  $t_i(u)$:

\begin{equation} \nonumber
B(t_i(u))(v) = \hbar \sum_{r \geq 0} t_{i, r} \frac{v^r}{r!} \in Y^0[[v]]. \label{7270}
\end{equation}

\begin{lm}
For arbitrary $i, j \in I$ we have the following equality
\begin{equation}
[B(t_i(u))(v), x^{\pm}_{j,s}] = \pm \frac{q_i^{a_{i,j}v} - q_i^{-a_{i,j}v}}{v}e^{\sigma^{\pm}_j v}x^{\pm}_{j,s}. \label{7271}
\end{equation}
\end{lm}
\begin{proof}
For simplicity, let $a = \pm \hbar d_i a_{i,j}/2$, $e^a = q_i^{\pm a_{i,j}}$. Then, from (\ref{7263}), we have
$$h_i(u) x^{\pm}_{j,s} h_i(u)^{-1} = \dfrac{u - \sigma^{\pm}_j + a}{u - \sigma^{\pm}_j - a}x^{\pm}_{j,s}.$$
Hence we obtain that
$$[t_i(u), x^{\pm}_{j,s}] = \log\left(\dfrac{u - \sigma^{\pm}_j + a}{u - \sigma^{\pm}_j - a}\right)x^{\pm}_{j,s}. $$
Using the above mentioned  (\ref{eq:Bor_cond})  equality
$$B(\log(1 - p u^{-1})) = \dfrac{1 - e^{pv}}{v} $$
we obtain 
$$[B(t_i(u))(v), x^{\pm}_{j,s}] = B(\log(1 - (\sigma^{\pm}_j - a) u^{-1}) - \log(1 - (\sigma^{\pm}_j + a) u^{-1}))x^{\pm}_{j,s} = $$
$$\left(\dfrac{1 - e^{(T^{\pm}_j - a)v}}{v} - \dfrac{1 - e^{(T^{\pm}_j + a)v}}{v}\right)x^{\pm}_{j,s} =  \dfrac{e^{av} - e^{-av}}{v}e^{\sigma^{\pm}_jv}x^{\pm}_{j,s}.$$
Lemma is proved.
\end{proof} 
 
To prove the theorem it remains to prove the equivalence of the relation (\ref{Y6}) to the relation (\ref{QL6}) and relations (\ref{Y7}) to the relations (\ref{QL7}), (\ref{QL71}).

\begin{predl}
1) The relation (\ref{Y6}) is equivalent to the relation (\ref{QL6}). \\
2)The relations (\ref{Y7}) are equivalent to the relations (\ref{QL7}), (\ref{QL71}).
\end{predl}

\begin{proof}
The proof  of item 1 is somewhat complicated, based on the formulas proved above and the Poincare-Birkhoff-Witt theorem for the Yangian (see, for example, \cite{St}, and also \cite{St12} for the more complicated case of the Yangian of the strange Lie superalgebra).

We briefly describe the scheme of the proof of item 2 of the proposition. \\ 
As shown above,

$$\Phi(E_{m, r})\Phi(E_{m+1, 0}) - q \Phi(E_{m+1, 0})\Phi(E_{m, r})$$ 
$$= e^{r\sigma^+_{m}}g_m(\sigma^+_m)\lambda^+_m(\sigma^+_m)(g_{m+1}(\sigma^+_{m+1})) x^+_{m,0}x^+_{m+1} $$
$$- q e^{r\sigma^+_{m}}g_{m+1}(\sigma^+_{m+1})\lambda^+_{m+1}(\sigma^+_{m+1})(g_{m}(\sigma^+_{m})) x^+_{m+1,0}x^+_{m}   $$
$$= e^{r\sigma^+_{m}}g_m(\sigma^+_m)\lambda^+_m(\sigma^+_m)(g_{m+1}(\sigma^+_{m+1}))(x^+_{m,0}x^+_{m+1})$$ 
$$- q \dfrac{e^{\sigma^+_{m+1}} - e^{\sigma^+_{m} + a\hbar}}{e^{\sigma^+_{m}} - e^{\sigma^+_{m+1} + a\hbar}}) \dfrac{\sigma^+_{m} - \sigma^+_{m+1} - a\hbar}{\sigma^+_{m+1} - \sigma^+_{m} + a\hbar}(x^+_{m+1,0}x^+_{m}) $$
$$= e^{r\sigma^+_{m}}g_m(\sigma^+_m)\lambda^+_m(\sigma^+_m)(g_{m+1}(\sigma^+_{m+1}))([x^+_{m,0}, x^+_{m+1,0}]).  $$
Similarly, it is checked that 
$$\Phi(E_{m+1, 0})\Phi(E_{m+2, s}) - q \Phi(E_{m+2, s})\Phi(E_{m+1, 0})$$ 
$$=  e^{s\sigma^+_{m+2}}g_{m+1}(\sigma^+_{m+1})\lambda^+_{m+1}(\sigma^+_{m+1})(g_{m+2}(\sigma^+_{m+2}))([x^+_{m+1,0}, x^+_{m+2,0}]). $$
Calculating the commutator of the obtained expressions, we see that we obtain

\begin{equation}\label{eq:4term}
A(g_{m}(\sigma^+_{m}), g_{m+1}(\sigma^+_{m+1}), g_{m+2}(\sigma^+_{m+2})) ([[x^+_{m,r}, x^+_{m+1,0}], [x^+_{m+1,0}, x^+_{m+2,s}]]), 
\end{equation}
and expression   $A(g_{m}(\sigma^+_{m}), g_{m+1}(\sigma^+_{m+1}), g_{m+2}(\sigma^+_{m+2}))$  is invertible. 

Thus, the expression (\ref{eq:4term}) vanishes if and only if the relation (\ref{Y7}) holds, if the plus sign is chosen in it. But the inverse of the expression (\ref{QL7}) as follows from the above calculations is equivalent to the relation (\ref{Y7}).

Checking the equivalence of the relation (\ref{QL71}) and the relation (\ref{Y7}) in the case when the minus sign is selected in it is carried out similarly. 

The proposition is proved. 

\end{proof}

The main result, theorem \ref{th:main}, is deduced from the assertions proved above. 

\begin{proof}

We briefly describe the plan of the proof of the theorem \ref{th:main}.  It is easy to see that the assertion of the existence of the monomorphism $\Phi$ defined by the formula (\ref{eq:mr}) follows from the propositions proved above. The assertion that this mapping can be uniquely extended to an isomorphism (\ref{eq:mr1}) is proved by standard arguments analogous to those given in \cite{G-TL}.

Namely, we first prove the isomorphism of the completions of the loop and current superalgebras, that is, the correctness of the assertion to be proved in the case when the deformation parameter $\hbar = 0$. Then, using the PBW theorem and the fact that the Yangian and the quantum loop superalgebra are flat deformations corresponding to the current and loop superbialgebras, we deduce from this the fact that the map $\widehat{\Phi}$ is an isomorphism. 

\end{proof}

\vspace{0.5cm}

The work is supported by grant "Domestic Research in Moscow"  of Interdisciplinary Scientific Center J.-V. Poncelet (CNRS UMI 2615) and Scoltech Center for Advanced Study.

\vspace{1.0cm}

Interdisciplinary Scientific Center J.-V. Poncelet (CNRS UMI 2615) and Scoltech Center for Advanced Study,   Don State Technical University, South Mathematical Institute. \\

E-mail address: stukopin@gmail.com

\end{document}